\documentclass[english]{amsart}

\usepackage[utf8]{inputenc}
\usepackage[english]{babel}
\usepackage[margin=1.15in]{geometry}

\usepackage{amsmath,amssymb,mathrsfs}
\usepackage[numbers]{natbib}

\numberwithin{equation}{section}

\newtheorem{theorem}{Theorem}[section]
\newtheorem{proposition}[theorem]{Proposition}
\newtheorem{lemma}[theorem]{Lemma}
\newtheorem{corollary}[theorem]{Corollary}

\theoremstyle{definition}

\theoremstyle{remark}
\newtheorem{remark}[theorem]{Remark}

\newcommand{\G}{\mathbb G}
\newcommand{\W}{\mathbf W}
\newcommand{\Mcal}{\mathbf M}
\newcommand{\supp}{\operatorname{supp}}
\newcommand{\diam}{\operatorname{diam}}
\newcommand{\dd}{\,\mathrm d}

\usepackage[
  colorlinks=true,
  linkcolor=blue,
  citecolor=blue,
  urlcolor=blue
]{hyperref}

\date{}

\begin{document}
\title{Quasilinear Equations with Exponential Reactions and Measure Data
on Carnot Groups}
\author{Shiguang Ma \, Zijian Wang}
\address[Shiguang Ma]
	{School of Mathematical Science and LPMC, Nankai University, Tianjin 300071, People's Republic of China.}
	\email{msgdyx8741@nankai.edu.cn}
	
	\address[Zijian Wang]
	{Chern Institute of Mathematics and LPMC, Nankai University, Tianjin 300071, People's Republic of China.}
	\email{wzj@mail.nankai.edu.cn}
\begin{abstract}
\setlength{\emergencystretch}{1em}
In this paper, we establish existence and pointwise Wolff potential bounds for nonnegative solutions of $-\Delta_{\mathbb{G},p}u=H_l(\alpha u^\beta)+\mu$ on Carnot groups, where $H_l(t)=e^t-\sum_{j=0}^{l-1}t^j/j!$ and $\mu$ is a nonnegative Radon measure. Under suitable conditions on the reaction parameters, we treat arbitrary open sets in the subcritical range $1<p<Q$ under a maximal potential condition, and bounded domains in the critical case $p=Q$ under a small total mass condition.
\end{abstract}
\maketitle

\section{Introduction}

\subsection{Carnot groups and the horizontal p-sub-Laplacian}
We first recall the geometric framework; standard references are
\cite{Folland1975,FollandStein1982,BLU}. Let $\G$ be a Carnot group of step
$s$. Its Lie algebra has a stratification
\begin{equation}
 \mathfrak g=V_1\oplus\cdots\oplus V_s,
 \quad [V_1,V_j]=V_{j+1}\ (1\le j<s),\quad [V_1,V_s]=\{0\}.
 \label{eq:stratification}
\end{equation}
Put $m=\dim V_1$ and choose a basis $X_1,\dots,X_m$ of $V_1$, identified
with the corresponding left-invariant horizontal vector fields. In
exponential coordinates, $\G$ is identified with $\mathbb R^N$, and the
automorphic dilations are
\begin{equation}
 \delta_r\bigl(x^{(1)},\dots,x^{(s)}\bigr)
 =\bigl(rx^{(1)},r^2x^{(2)},\dots,r^sx^{(s)}\bigr),\quad r>0,
 \label{eq:dilation}
\end{equation}
where $x^{(j)}\in V_j$. The homogeneous dimension is
\begin{equation}
 Q=\sum_{j=1}^s j\dim V_j.
 \label{eq:homogeneous-dimension}
\end{equation}

The Carnot--Carath\'eodory distance $d$ is defined by declaring the fields
$X_1,\dots,X_m$ orthonormal. It is left invariant and homogeneous:
\begin{equation}
 d(g\circ x,g\circ y)=d(x,y),\quad
 d(\delta_rx,\delta_ry)=r\,d(x,y).
 \label{eq:cc-properties}
\end{equation}
Lebesgue measure in exponential coordinates is Haar measure. In particular,
if $B_r(x)=\{y:d(x,y)<r\}$, then
\begin{equation}
 |B_r(x)|=|B_1(e)|r^Q,\quad x\in\G,\ r>0.
 \label{eq:ball-volume}
\end{equation}
Thus $(\G,d,dx)$ is a geodesic doubling metric measure space. We fix this
structure throughout the paper; structural constants may depend on it.

For a smooth function $u$, write
\[
 Xu=(X_1u,\dots,X_mu),\quad
 |Xu|=\left(\sum_{j=1}^m|X_ju|^2\right)^{1/2},
\]
and define the horizontal $p$-sub-Laplacian by
\begin{equation}
 \Delta_{\G,p}u=\sum_{j=1}^mX_j\bigl(|Xu|^{p-2}X_ju\bigr),
 \quad 1<p<\infty.
 \label{eq:plaplacian}
\end{equation}

\subsection{Research background and main results}

Classical existence results for quasilinear equations with measure data
include the work of Boccardo--Gallou\"et and
Boccardo--Gallou\"et--Orsina
\cite{BoccardoGallouet1989,BoccardoGallouetOrsina1996}; see also the
monograph \cite{MarcusVeron2014}. Pointwise estimates in terms of nonlinear
potentials play an important role in this theory; see
\cite{KTM-ASNSP,KM2,AdamsHedberg1996,KuusiMingione2014}.

For equations with power reactions, Phuc and Verbitsky obtained necessary
and sufficient conditions for solvability in the Euclidean setting
\cite{PhucVerbitsky2008} and on Carnot groups \cite{PVTAMS}. Their approach
relates the measure data to nonlinear potential estimates and capacitary
conditions. The Carnot group theory in \cite{PVTAMS} also provides
existence and Wolff potential bounds for the measure data Dirichlet
problem, together with a regularized construction of monotone iterates.
These results and the weak continuity theorem of Trudinger and Wang
\cite{TW} provide the measure data framework used below.

Exponential reactions in the Euclidean subcritical range $1<p<N$ were
studied by Nguyen and V\'eron \cite{NGARMA2014}, with existence conditions
expressed through logarithmically corrected maximal potentials. In the
critical case $p=N$, Ma and Wang \cite{MaWang2026} established existence
for the Dirichlet problem with truncated exponential reaction on bounded
domains under a smallness condition on the total mass of the measure
data. Thus, in their critical result, no maximal potential smallness
condition is required. Related endpoint exponential estimates include
the Brezis--Merle estimate \cite{Brezis-Merle} and the Moser--Trudinger
inequality on Carnot groups proved by Balogh, Manfredi, and Tyson
\cite{BaloghManfrediTyson2003}.

For $l\in\mathbb N^*$, $\alpha>0$, and $\beta\geq1$, put
\[
 H_l(t)=e^t-\sum_{j=0}^{l-1}\frac{t^j}{j!},\quad
 P_{l,\alpha,\beta}(t)=H_l(\alpha t^\beta),\quad t\geq0.
\]
For a nonnegative Radon measure $\nu$ on $\G$ and $0<R\leq\infty$, define
\[
 \W_{1,p}^R[\nu](x)=\int_0^R
 \left(\frac{\nu(B_t(x))}{t^{Q-p}}\right)^{\!1/(p-1)}
 \frac{\dd t}{t}.
\]
For $\eta\geq0$, let
\[
 h_\eta(t)=
 \begin{cases}
 (-\log t)^{-\eta},&0<t<\frac12,\\
 (\log2)^{-\eta},&t\geq\frac12,
 \end{cases}
 \quad
 \Mcal_{p,R}^{\eta}[\nu](x)=
 \sup_{0<t\leq R}\frac{\nu(B_t(x))}{t^{Q-p}h_\eta(t)}.
\]
We write $\Mcal_p^\eta=\Mcal_{p,\infty}^\eta$.
When these potentials are applied to a finite measure defined on an open
subset of $\G$, the measure is extended by zero outside that set.
For a nonnegative density $f$, we write $\W_{1,p}^R[f]$ for
$\W_{1,p}^R[f\,dx]$.

We denote by $S^{1,p}(\Omega)$ the horizontal Sobolev space of functions
$v\in L^p(\Omega)$ whose distributional horizontal derivatives
$X_jv$ belong to $L^p(\Omega)$ for $1\leq j\leq m$.
The space $S_0^{1,p}(\Omega)$ is the closure of $C_c^\infty(\Omega)$ in
$S^{1,p}(\Omega)$; see \cite[Section~3]{PVTAMS}.
Let $\nu$ be a nonnegative Radon measure on an open set $\Omega$. We call
a nonnegative function $u$ a solution of $-\Delta_{\G,p}u=\nu$ with zero
boundary values if it is $p$-superharmonic, its truncations satisfy
\[
 \zeta(u\wedge k)\in S_0^{1,p}(\Omega)
 \quad
 \text{for every }k>0
 \text{ and every }\zeta\in C_c^\infty(\G),
\]
and, with $Xu$ denoting the generalized horizontal gradient
$Xu=\lim_{n\to\infty}X(u\wedge n)$ almost everywhere,
\[
 \int_\Omega |Xu|^{p-2}Xu\cdot X\varphi\,dx
 =\int_\Omega\varphi\,d\nu
 \quad\text{for every }\varphi\in C_c^\infty(\Omega).
\]
For bounded $\Omega$, the truncation condition above is
equivalent to $u\wedge k\in S_0^{1,p}(\Omega)$ for every $k>0$.
For $\Omega=\G$, it is equivalent to
$u\wedge k\in S_{\mathrm{loc}}^{1,p}(\G)$ for every $k>0$.Throughout the paper, $C$ denotes a positive constant that may change from line to line. 

The first theorem applies to arbitrary open subsets of $\G$, including
$\G$ itself. 
\begin{theorem}\label{thm:subcritical}
 Let $1<p<Q$, $\alpha>0$, $\beta\geq1$, and
 \[
 l\beta>\frac{Q(p-1)}{Q-p},\quad
 \eta=\frac{(p-1)(\beta-1)}{\beta}.
 \]
 Let $\Omega\subset\G$ be any nonempty open set and let $\mu$ be a
 nonnegative finite Radon measure with compact support
 $K_\mu\Subset\Omega$.
 Choose a ball $B_R(x_0)$ such that $K_\mu\subset B_R(x_0)$.
 There exist two constants
 \[
 b=b(\G,p,l,\alpha,\beta,R)>0,
 \quad
 \varepsilon_*=\varepsilon_*(\G,p,l,\alpha,\beta,R,C_*)>0,
 \]
 such that if
 \begin{equation}
\|\Mcal_p^\eta[\mu]\|_{L^\infty(\G)}\leq\varepsilon_*
 \label{eq:sub-small}
 \end{equation}
 then the problem
 \begin{equation}
 -\Delta_{\G,p}u=P_{l,\alpha,\beta}(u)+\mu
 \quad\text{in }\Omega,
 \label{eq:subproblem}
 \end{equation}
 admits a nonnegative solution $u$ in the sense defined above. With $\omega=\mu+b\chi_{B_{2R}(x_0)}dx$, this solution satisfies
 \begin{equation}
P_{l,\alpha,\beta}\bigl(2c_pC_*\W_{1,p}^\infty[\omega]\bigr)
 \in L^1(\G),\quad
 0\leq u\leq2c_pC_*\W_{1,p}^\infty[\omega]
 \quad\text{in }\Omega,
 \label{eq:subbound}
 \end{equation}
 where $c_p=1\vee2^{(2-p)/(p-1)}$ and
 $C_*=C_*(\G,p)$ is the constant in the measure data Wolff potential
 estimate. If $\Omega=\G$, then $\inf_{\G}u=0$.
\end{theorem}

The second theorem is a Carnot group analogue of the Euclidean critical
$N$-Laplacian existence result of Ma and Wang \cite{MaWang2026}.
Here the reaction is $H_l(\alpha u)$, corresponding to $\beta=1$, and the
smallness condition involves only the total mass of the measure data.

\begin{theorem}\label{thm:critical}
 Assume that $Q>1$. Let $\Omega\Subset\G$ be a bounded open set, let $\alpha>0$, and let
 $l\in\mathbb N^*$ satisfy $l>Q-1$. There exists
$\varepsilon_0=\varepsilon_0(\G,Q,l,\alpha,\diam\Omega)>0$ such that
 every nonnegative finite Radon measure $\mu$ in $\Omega$ satisfying
 \begin{equation}
 \mu(\Omega)\leq\varepsilon_0
 \label{eq:critical-small}
 \end{equation}
 admits a nonnegative solution of
 \begin{equation}
 \begin{cases}
 -\Delta_{\G,Q}u=H_l(\alpha u)+\mu&\text{in }\Omega,\\
 u=0&\text{on }\partial\Omega
 \end{cases}
 \label{eq:criticalproblem}
 \end{equation}
 in the sense defined above. More precisely, for a ball
 $B_{10D}(x_0)\supset\Omega$, where $D=\diam\Omega$, one can choose
 $b=b(\G,Q,l,\alpha,D)>0$ so that, for
 $\omega=\mu+b\chi_{B_{10D}(x_0)}dx$,
 \begin{equation}
 H_l\bigl(2\alpha C_*\W_{1,Q}^{2D}[\omega]\bigr)\in L^1(\Omega),
 \quad
 0\leq u\leq2C_*\W_{1,Q}^{2D}[\omega]\quad\text{in }\Omega.
 \label{eq:criticalbound}
 \end{equation}
\end{theorem}

\begin{remark}
 The background measures in \eqref{eq:subbound} and
 \eqref{eq:criticalbound} are finite and compactly supported.
 For the whole space $L^1$ assertion in \eqref{eq:subbound}, the background
 cannot be replaced by a positive constant density on all of $\G$,
 even at a fixed finite truncation scale, its Wolff potential is a
 positive constant, and the associated reaction is not integrable on
 $\G$.
\end{remark}
Our results extend the existence theory for exponential reaction equations in \cite{NGARMA2014,MaWang2026} to Carnot groups of
arbitrary step. The main contribution is to establish nonlinear potential estimates that control the exponential reaction in this setting. In the subcritical case, these estimates provide global integrability of the reaction and bounds independent of the approximating domains. This allows us to treat arbitrary open sets, including unbounded proper subsets, beyond the bounded domain and whole space settings considered in \cite{NGARMA2014}.
\section{Potential theory of the p-sub-Laplacian}

\subsection{Whitney decomposition}

For a Borel set $E$, $\nu\restriction E$ denotes the restriction of a
measure $\nu$ to $E$.  The elementary inequality
\begin{equation}
 \W_{1,p}^R[\nu_1+\nu_2]
 \le c_p\bigl(\W_{1,p}^R[\nu_1]+\W_{1,p}^R[\nu_2]\bigr),
 \quad c_p=1\vee2^{(2-p)/(p-1)},
 \label{eq:Wsubadditive}
\end{equation}
will be used repeatedly.

We recall some facts from the measure data theory on Carnot groups.  In the
form needed here, they are contained in \cite{PVTAMS,TW}.  If $D\Subset\G$
is bounded and $\nu$ is a finite nonnegative Radon measure in $D$, there is a
nonnegative $p$-superharmonic solution $v$ of
\begin{equation}
 -\Delta_{\G,p}v=\nu\quad\hbox{in }D,\quad v=0\quad\hbox{on }\partial D,
 \label{eq:linear-dirichlet}
\end{equation}
and
\begin{equation}
  v(x)\le C_*\W_{1,p}^{2\diam D}[\nu](x),
  \quad x\in D.
 \label{eq:linear-wolff}
\end{equation}

We shall need the following form of the Whitney decomposition.  It is a
standard consequence of the dyadic Whitney covering lemma in geometrically
doubling quasi-metric spaces; see \cite[Lemma 4.2]{Covering}.  We include the
statement with the precise properties used below.

\begin{lemma}\label{lem:whitney}
 Let $U\subsetneq\G$ be a nonempty open set.  There are constants
 $\Lambda>1$ and $N\in\mathbb N$, depending only on the doubling structure
 of $(\G,d)$, and a countable family of balls
 $B_j=B_{r_j}(x_j)$ such that
 \begin{equation}
 U=\bigcup_jB_j,\quad 5B_j\subset U,
 \quad \sum_j\chi_{5B_j}\le N.
 \label{eq:whitney-cover}
 \end{equation}
 Moreover, for every $j$ there is a point $z_j\in U^c$ for which
 \begin{equation}
 d(x_j,z_j)\le\Lambda r_j.
 \label{eq:whitney-boundary}
 \end{equation}
 The radii may be chosen as a fixed multiple of dyadic numbers, and if $5B_i\cap5B_j\ne\varnothing$,
 then $r_i/r_j$ is bounded above and below by structural constants.
\end{lemma}
\begin{proof}
Let $\delta(x)=d(x,U^c)$ for $x\in U$.
Since $U$ is open, $\delta(x)>0$, and
\[
|\delta(x)-\delta(y)|\le d(x,y),
\quad x,y\in U.
\]
For each $x\in U$, choose the largest dyadic number
$r(x)=2^k$ such that $r(x)\le\delta(x)/100$.
Then
\[
100r(x)\le\delta(x)<200r(x).
\]
Choose a maximal pairwise disjoint subfamily
$
B_j^0=B_{r_j^0}(x_j),
r_j^0=r(x_j)
$
of $\{B_{r(x)}(x):x\in U\}$.
Since $\G$ is separable, this family is countable.

We first show that the balls $5B_j^0$ cover $U$.
For any $x\in U$, maximality gives an index $j$ such that
$B_{r(x)}(x)\cap B_j^0\ne\varnothing$.
Hence $d(x,x_j)<r(x)+r_j^0$, and
\[
100r(x)\le\delta(x)
\le\delta(x_j)+d(x,x_j)
<200r_j^0+r(x)+r_j^0.
\]
It follows that $r(x)<3r_j^0$ and
$d(x,x_j)<4r_j^0$. Thus
\[
U\subset\bigcup_j5B_j^0.
\]
Set $r_j=5r_j^0$ and $B_j=B_{r_j}(x_j)$.
The choice of $r_j^0$ gives
\[
20r_j\le\delta(x_j)<40r_j.
\]
Thus $5B_j\subset U$, and hence $\bigcup_jB_j\subset U$.
Combining this with the preceding covering
$U\subset\bigcup_j5B_j^0=\bigcup_jB_j$, we obtain
\[
U=\bigcup_jB_j.
\]
By the definition of $\delta(x_j)$, we can also choose
$z_j\in U^c$ such that
\[
d(x_j,z_j)<40r_j.
\]

We next compare $r_i$ and $r_j$ when
$5B_i\cap5B_j\ne\varnothing$, in this case,
$d(x_i,x_j)<5r_i+5r_j$. Therefore
\[
20r_i\le\delta(x_i)
\le\delta(x_j)+d(x_i,x_j)
<40r_j+5r_i+5r_j,
\]
which gives $r_i<3r_j$.
Interchanging $i$ and $j$, we obtain
\[
\frac13r_j<r_i<3r_j.
\]
Finally, we show that each point belongs to at most a fixed
number of the balls $5B_j$.
Fix $x\in U$. Since $U=\bigcup_jB_j$, we can choose $j_0$
such that $x\in B_{j_0}\subset5B_{j_0}$.
For every $j$ with $x\in5B_j$, both $5B_j$ and $5B_{j_0}$
contain $x$, so their intersection is nonempty.
The preceding comparison therefore gives
\[
\frac13r_{j_0}<r_j<3r_{j_0}.
\]
Since $r_j^0=r_j/5$, the radius of the original ball $B_j^0$
satisfies
$
r_j^0=\frac{r_j}{5}>\frac{r_{j_0}}{15}.
$
Moreover, for $y\in B_j^0$, we have
$d(y,x_j)<r_j/5$ and $d(x_j,x)<5r_j$. Hence
\[
\begin{aligned}
d(y,x)
\le d(y,x_j)+d(x_j,x)
<\frac{r_j}{5}+5r_j
=\frac{26}{5}r_j
<\frac{78}{5}r_{j_0}<16r_{j_0}.
\end{aligned}
\]
Thus, for every $j$ with $x\in5B_j$,
$
B_j^0\subset B_{16r_{j_0}}(x),
|B_j^0|\ge |B_{r_{j_0}/15}(e)|.
$
The original balls $B_j^0$ were chosen to be pairwise disjoint.
Adding their volumes, we obtain
\[
\begin{aligned}
|B_{r_{j_0}/15}(e)|
\sum_{\{j:x\in5B_j\}}1
\le\sum_{\{j:x\in5B_j\}}|B_j^0|
=\left|\bigcup_{\{j:x\in5B_j\}}B_j^0\right|
\le |B_{16r_{j_0}}(x)|.
\end{aligned}
\]
The sum on the left counts the balls $5B_j$ containing $x$.
Therefore, by \eqref{eq:ball-volume},
\[
\sum_j\chi_{5B_j}(x)
=\sum_{\{j:x\in5B_j\}}1
\le
\frac{|B_{16r_{j_0}}(x)|}{|B_{r_{j_0}/15}(e)|}
=240^Q.
\]
For $x\notin U$, this sum is zero because $5B_j\subset U$
for every $j$. Thus the overlap bound holds on all of $\G$.
Since each $r_j=5r_j^0$ is five times a dyadic number,
\eqref{eq:whitney-cover} and \eqref{eq:whitney-boundary}
follow with $N=240^Q$ and $\Lambda=40$.
\end{proof}
\subsection{The good-lambda estimate}

In this subsection, we first establish a good-$\lambda$ estimate relating
$\W_{1,p}^{\infty}[\nu]$ to $\Mcal_p^\eta[\nu]$.
We then deduce the local exponential integrability estimate
in Corollary~\ref{cor:localexp}.

\begin{proposition}\label{prop:goodlambda}
 Let $1<p<Q$, $0\le\eta<p-1$, and let $\nu$ be a nonnegative finite
 measure with $\supp\nu\subset B_R(x_0)$.  There are constants
  $A>1$, $C>0$, $c>0$, and $\epsilon_0>0$, depending only on
  $\G,p,\eta,R$, the fixed cutoff $1/2$ in $h_\eta$, and the Whitney
  constants $\Lambda,N$ in Lemma~\ref{lem:whitney}, such that, whenever
 \begin{equation}
 \lambda\ge C\nu(\G)^{1/(p-1)}R^{-(Q-p)/(p-1)},
 \quad 0<\epsilon\le\epsilon_0,
 \label{eq:lambda-range}
 \end{equation}
 one has
 \begin{equation}
 \begin{split}
 &\left|\left\{\W_{1,p}^{\infty}[\nu]>A\lambda,\
 (\Mcal_{p}^{\eta}[\nu])^{1/(p-1)}\le\epsilon\lambda\right\}\right|\\
 &\quad\le C\exp\!\left(-c\epsilon^{-
 \frac{p-1}{p-1-\eta}}\right)
 \left|\left\{\W_{1,p}^{\infty}[\nu]>\lambda\right\}\right|.
 \end{split}
 \label{eq:goodlambda}
 \end{equation}
\end{proposition}

\begin{proof}
 Put $s=1/(p-1)$, $\theta=(p-1-\eta)/(p-1)$, and
 \[
 E_\lambda=\{\W_{1,p}^{\infty}[\nu]>\lambda\}.
 \]
 The potential is lower semicontinuous, so $E_\lambda$ is open.  If
 $x\notin B_{2R}(x_0)$, then $B_t(x)\cap\supp\nu=\varnothing$ for
 $t<d(x,x_0)-R$, and therefore
 \begin{equation}
 \W_{1,p}^{\infty}[\nu](x)
 \le \frac{p-1}{Q-p}\nu(\G)^s
 (d(x,x_0)-R)^{-(Q-p)/(p-1)}.
 \label{eq:far-level}
 \end{equation}
 By \eqref{eq:lambda-range}, $E_\lambda$ is contained in a ball
 $B_{C_R}(x_0)$; in particular, it has finite measure and is a proper open
 subset of $\G$. If $E_\lambda=\varnothing$, the conclusion is immediate. We therefore assume that $E_\lambda\neq\varnothing$. Apply Lemma~\ref{lem:whitney} to fix a Whitney decomposition
 $E_\lambda=\bigcup_jB_j$ with
 $B_j=B_{r_j}(x_j)$.  We use the version for which the dilates $5B_j$ have
 bounded overlap and there is $z_j\in E_\lambda^c$ satisfying
 \begin{equation}
 d(z_j,x_j)\le\Lambda r_j.
 \label{eq:whitney-point}
 \end{equation}
 Let $\rho_j=(\Lambda+1)r_j$.  If $x\in B_j$ and $t\ge\rho_j$, then
 $d(x,z_j)\le\rho_j\le t$, hence
 \[
 B_t(x)\subset B_{2t}(z_j).
 \]
 A change of variables $\tau=2t$ gives
 \begin{equation}
 \begin{split}
 \int_{\rho_j}^{\infty}
 \left(\frac{\nu(B_t(x))}{t^{Q-p}}\right)^s\frac{\dd t}{t}
 &\le 2^{2(Q-p)/(p-1)}
 \W_{1,p}^{\infty}[\nu](z_j)\\
 &\le C_1\lambda.
 \end{split}
 \label{eq:tail-comparison}
 \end{equation}
 Choose $A=C_1+2$, It follows that
 \begin{equation}
 \begin{split}
 F_{\epsilon,\lambda}\cap B_j
 &\subset E_{j,\epsilon,\lambda},\\
 F_{\epsilon,\lambda}
 &=\{\W_{1,p}^{\infty}[\nu]>A\lambda,
 (\Mcal_p^\eta[\nu])^s\le\epsilon\lambda\},\\
 E_{j,\epsilon,\lambda}
 &=\left\{x\in B_j:\W_{1,p}^{\rho_j}[\nu](x)>\lambda,
 (\Mcal_p^\eta[\nu](x))^s\le\epsilon\lambda\right\}.
 \end{split}
 \label{eq:localized-set}
 \end{equation}

 We next estimate $|E_{j,\epsilon,\lambda}|$.  If this set is empty there
 is nothing to prove, so choose $y_j\in E_{j,\epsilon,\lambda}$.  Since
 $E_\lambda\subset B_{C_R}(x_0)$, the Whitney radii satisfy
 $\rho_j\le C_R$.  Let $m_0$ be an integer such that
 $2^{-m_0}\rho_j\le1/2$ for every $j$.  For $m\ge m_0$ and
 $x\in E_{j,\epsilon,\lambda}$, the definition of the maximal operator
 gives
 \begin{equation}
 \begin{split}
 \int_{2^{-m}\rho_j}^{\rho_j}
 \left(\frac{\nu(B_t(x))}{t^{Q-p}}\right)^s\frac{\dd t}{t}
 \le\epsilon\lambda
 \int_{2^{-m}\rho_j}^{\rho_j}h_\eta(t)^s\frac{\dd t}{t}
 \le C_R\epsilon\lambda+C\epsilon\lambda m^\theta.
 \end{split}
 \label{eq:large-scales}
 \end{equation}
 Here the first constant accounts for the interval where $t\ge1/2$; on
 $(0,1/2)$ the integral is computed explicitly from
 \[
 \int(-\log t)^{-\eta/(p-1)}\frac{\dd t}{t}
 =-\frac{p-1}{p-1-\eta}
 (-\log t)^\theta.
 \]
Define
 \[
 g_i(x)=\int_{2^{-i}\rho_j}^{2^{-i+1}\rho_j}
 \left(\frac{\nu(B_t(x))}{t^{Q-p}}\right)^s\frac{\dd t}{t},
 \quad i\ge m+1.
 \]
 It follows from \eqref{eq:large-scales} that
 \begin{equation}
 E_{j,\epsilon,\lambda}\subset
 \left\{x\in B_j:\sum_{i=m+1}^{\infty}g_i(x)>L_m\right\},
 \quad L_m=\lambda\bigl(1-C_R\epsilon-C\epsilon m^\theta\bigr).
 \label{eq:tail-level}
 \end{equation}
 We now establish the estimate for one dyadic piece.  Monotonicity of
 $t\mapsto\nu(B_t(x))$, together with $t\ge2^{-i}\rho_j$ on the
 interval of integration, gives
 \begin{equation}
 g_i(x)^{p-1}\le C
 \frac{\nu(B_{2^{-i+1}\rho_j}(x))}{(2^{-i}\rho_j)^{Q-p}}.
 \label{eq:dyadic-jensen}
 \end{equation}
Recall that $y_j\in B_j=B_{r_j}(x_j)$ and
$\rho_j=(\Lambda+1)r_j$.
For $x\in B_j$, we have $d(x,y_j)<2r_j$.
Since $2^{-i+1}\rho_j\le\rho_j$, the triangle inequality gives
\[
B_{2^{-i+1}\rho_j}(x)
\subset B_{2^{-i+1}\rho_j+2r_j}(y_j)
\subset B_{3\rho_j}(y_j),
\]
Thus, if $z\notin B_{3\rho_j}(y_j)$, then
\[
B_j\cap B_{2^{-i+1}\rho_j}(z)=\varnothing.
\]
This allows us to restrict the integral in $z$ to
$B_{3\rho_j}(y_j)$ after changing the order of integration.
 Consequently, Fubini's theorem and \eqref{eq:ball-volume} yield
\begin{equation}
\begin{split}
\int_{B_j}g_i(x)^{p-1}\,dx
&\le \frac{C}{(2^{-i}\rho_j)^{Q-p}}
\int_{B_j}\nu(B_{2^{-i+1}\rho_j}(x))\,dx\\
&=\frac{C}{(2^{-i}\rho_j)^{Q-p}}
\int_{B_{3\rho_j}(y_j)}
|B_j\cap B_{2^{-i+1}\rho_j}(z)|\,d\nu(z)\\
&\le C\frac{(2^{-i}\rho_j)^Q}{(2^{-i}\rho_j)^{Q-p}}
\nu(B_{3\rho_j}(y_j))\\
&\le C2^{-ip}\rho_j^p
(3\rho_j)^{Q-p}h_\eta(3\rho_j)(\epsilon\lambda)^{p-1}\\
&\le C2^{-ip}\rho_j^Q(\epsilon\lambda)^{p-1}\\
&\le C2^{-ip}|B_j|(\epsilon\lambda)^{p-1}.
\end{split}
\label{eq:dyadic-Lp}
\end{equation}
 The last inequality follows because $h_\eta$ is bounded above and
 $\rho_j\asymp r_j$. We first choose $m$ so that $L_m\ge\lambda/2$.
Choose $\kappa>0$ such that $C\kappa^\theta\le1/4$, where
$C$ is the constant in the definition of $L_m$.
Decrease $\epsilon_0>0$ so that
\[
\epsilon_0\le1,\quad
C_R\epsilon_0\le\frac14,\quad
\kappa\epsilon_0^{-1/\theta}\ge2\max\{1,m_0\}.
\]
Then, for every $0<\epsilon\le\epsilon_0$, we can choose an
integer $m\ge m_0$ satisfying
\begin{equation}
\frac12\kappa\epsilon^{-1/\theta}
\le m\le\kappa\epsilon^{-1/\theta}.
\label{eq:choose-m}
\end{equation}
Since
\[
C\epsilon m^\theta
\le C\epsilon(\kappa\epsilon^{-1/\theta})^\theta
=C\kappa^\theta\le\frac14,
\]
we have
\[
L_m=\lambda(1-C_R\epsilon-C\epsilon m^\theta)
\ge\frac{\lambda}{2}>0.
\]
Fix $0<\sigma<p/(p-1)$, depending only on $p$, and set
\[
a_i=(1-2^{-\sigma})2^{-\sigma(i-m-1)},
\quad i\ge m+1.
\]
Then
\[
\sum_{i=m+1}^{\infty}a_i
=(1-2^{-\sigma})\sum_{k=0}^{\infty}2^{-\sigma k}
=1.
\]
If $g_i(x)\le a_iL_m$ for every $i\ge m+1$, then
$\sum_{i=m+1}^{\infty}g_i(x)\le L_m$.
Thus \eqref{eq:tail-level} gives
\[
E_{j,\epsilon,\lambda}
\subset
\bigcup_{i=m+1}^{\infty}
\{x\in B_j:g_i(x)>a_iL_m\}.
\]
By Chebyshev's inequality and \eqref{eq:dyadic-Lp},
\begin{equation}
\begin{aligned}
|E_{j,\epsilon,\lambda}|
&\le\sum_{i=m+1}^{\infty}
\frac{1}{(a_iL_m)^{p-1}}
\int_{B_j}g_i(x)^{p-1}\,dx\\
&\le C\left(\frac{\epsilon\lambda}{L_m}\right)^{p-1}
|B_j|\sum_{i=m+1}^{\infty}2^{-ip}a_i^{-(p-1)}\\
&=\frac{C2^{-(m+1)p}}{(1-2^{-\sigma})^{p-1}}
\left(\frac{\epsilon\lambda}{L_m}\right)^{p-1}
|B_j|\sum_{k=0}^{\infty}2^{[\sigma(p-1)-p]k}\\
&\le C2^{-(m+1)p}
\left(\frac{\epsilon\lambda}{L_m}\right)^{p-1}|B_j|.
\end{aligned}
\label{eq:local-goodlambda}
\end{equation}
Here we set $k=i-m-1$ in the third line.
The series converges because $\sigma(p-1)-p<0$. Since $L_m\ge\lambda/2$, $\epsilon\le1$ and
$m\ge\frac12\kappa\epsilon^{-1/\theta}$, it follows that
\begin{equation}
\begin{aligned}
|E_{j,\epsilon,\lambda}|
&\le C\epsilon^{p-1}2^{-(m+1)p}|B_j|\\
&\le C\exp\!\left(-c\epsilon^{-1/\theta}\right)|B_j|,
\end{aligned}
\label{eq:local-final}
\end{equation}
where $c=\frac12\kappa p\log2$. Finally, the Whitney balls cover $E_\lambda$ and have bounded
overlap. Summing \eqref{eq:local-final} and using
\eqref{eq:localized-set}, we obtain
\[
\begin{aligned}
|F_{\epsilon,\lambda}|
&\le\sum_j|E_{j,\epsilon,\lambda}|\\
&\le C\exp\!\left(-c\epsilon^{-1/\theta}\right)
\sum_j|B_j|\\
&\le C\exp\!\left(-c\epsilon^{-1/\theta}\right)|E_\lambda|.
\end{aligned}
\]
Since $1/\theta=(p-1)/(p-1-\eta)$, this proves
\eqref{eq:goodlambda}.
\end{proof}

\begin{corollary}\label{cor:localexp}
 Let $\eta=(p-1)(\beta-1)/\beta$.  Under the assumptions of Proposition
 \ref{prop:goodlambda}, if $\|\Mcal_p^\eta[\nu]\|_\infty\le L$, then
 there are constants $\delta=\delta(\G,p,\beta,R,L)>0$ and
 $C=C(\G,p,\beta,R,L)>0$ for which
 \begin{equation}
 \int_{B_R(x_0)}
 \exp\!\left(\delta[\W_{1,p}^{\infty}[\nu](x)]^\beta\right)dx
 \le C|B_R(x_0)|.
 \label{eq:local-exponential}
 \end{equation}
\end{corollary}
\begin{proof}
If $L=0$, then $\nu=0$ and the integral in
\eqref{eq:local-exponential} equals $|B_R(x_0)|$.
Assume now that $L>0$. We first determine how large $t$ must be to apply
Proposition~\ref{prop:goodlambda}.
To apply Proposition~\ref{prop:goodlambda} with $\lambda=t$,
we first verify the lower bound on $t$ in
\eqref{eq:lambda-range}.
Since $\supp\nu\subset B_R(x_0)$, the maximal potential bound gives
\[
\nu(\G)=\nu(B_R(x_0))
\le L R^{Q-p}h_\eta(R).
\]
It follows that
\[
C\nu(\G)^{1/(p-1)}R^{-(Q-p)/(p-1)}
\le C h_\eta(R)^{1/(p-1)}L^{1/(p-1)},
\]
where $C$ is the constant in \eqref{eq:lambda-range}.
Thus the required lower bound is satisfied whenever
\[
t\ge C h_\eta(R)^{1/(p-1)}L^{1/(p-1)}.
\]
We next bound the measure of
$\{\W_{1,p}^{\infty}[\nu]>t\}$, which appears on the
right hand side of \eqref{eq:goodlambda}.
For $x\notin B_{2R}(x_0)$, we have $d(x,x_0)-R\ge R$.
Hence \eqref{eq:far-level} and the preceding mass estimate give
\[
\begin{aligned}
\W_{1,p}^{\infty}[\nu](x)
&\le \frac{p-1}{Q-p}\nu(\G)^{1/(p-1)}
(d(x,x_0)-R)^{-(Q-p)/(p-1)}\\
&\le \frac{p-1}{Q-p}
\bigl(LR^{Q-p}h_\eta(R)\bigr)^{1/(p-1)}
R^{-(Q-p)/(p-1)}\\
&=\frac{p-1}{Q-p}
h_\eta(R)^{1/(p-1)}L^{1/(p-1)}.
\end{aligned}
\]
Therefore, if $t$ is at least the last quantity, then
$\W_{1,p}^{\infty}[\nu](x)\le t$ outside $B_{2R}(x_0)$.

Choose $K>0$, depending only on $\G,p,\beta,R$, such that
\[
K\ge h_\eta(R)^{1/(p-1)}
\max\left\{C,\frac{p-1}{Q-p}\right\},
\]
and set $t_0=KL^{1/(p-1)}$.
Then, for every $t\ge t_0$, the lower bound in
\eqref{eq:lambda-range} is satisfied and
\[
\{\W_{1,p}^{\infty}[\nu]>t\}\subset B_{2R}(x_0).
\]
In particular,
\[
\bigl|\{\W_{1,p}^{\infty}[\nu]>t\}\bigr|
\le |B_{2R}(x_0)|.
\]
We also choose $K\ge\epsilon_0^{-1}$. For $t\ge t_0$, set
\[
\epsilon=\frac{L^{1/(p-1)}}{t}.
\]
Then $0<\epsilon\le\epsilon_0$, and the maximal potential bound gives
\[
(\Mcal_p^\eta[\nu])^{1/(p-1)}
\le L^{1/(p-1)}=\epsilon t
\]
almost everywhere in $\G$.
Thus the maximal potential condition on the left hand side of
\eqref{eq:goodlambda} is satisfied.
Since
\[
\frac{p-1}{p-1-\eta}=\beta,
\]
applying \eqref{eq:goodlambda} yields
\[
\begin{aligned}
|\{\W_{1,p}^{\infty}[\nu]>At\}|
&\le C\exp\!\left(-ct^\beta L^{-\beta/(p-1)}\right)
|\{\W_{1,p}^{\infty}[\nu]>t\}|\\
&\le C\exp\!\left(-ct^\beta L^{-\beta/(p-1)}\right)
|B_{2R}(x_0)|\\
&\le C\exp\!\left(-ct^\beta L^{-\beta/(p-1)}\right)
|B_R(x_0)|.
\end{aligned}
\]
We now choose the coefficient in the exponential integral.
Set
\[
\delta=\frac{c}{2A^\beta}L^{-\beta/(p-1)}>0.
\]
Replacing $t$ by $t/A$ in the preceding estimate, we obtain
\[
|\{\W_{1,p}^{\infty}[\nu]>t\}|
\le C e^{-2\delta t^\beta}|B_R(x_0)|,
\quad t\ge At_0.
\]
We apply the layer cake formula and split the integral at $At_0$.
For $0<t<At_0$, we bound the measure of the level set in
$B_R(x_0)$ by $|B_R(x_0)|$.
For $t\ge At_0$, we use the exponential decay just proved.
It follows that
\[
\begin{aligned}
&\int_{B_R(x_0)}
\exp\!\left(\delta[\W_{1,p}^{\infty}[\nu](x)]^\beta\right)dx\\
&\quad=|B_R(x_0)|
+\delta\beta\int_0^\infty
t^{\beta-1}e^{\delta t^\beta}
|\{x\in B_R(x_0):\W_{1,p}^{\infty}[\nu](x)>t\}|
\,\dd t\\
&\quad\le |B_R(x_0)|
+\delta\beta|B_R(x_0)|
\int_0^{At_0}t^{\beta-1}e^{\delta t^\beta}\,\dd t
+C\delta\beta|B_R(x_0)|
\int_{At_0}^\infty t^{\beta-1}e^{-\delta t^\beta}\,\dd t\\
&\quad=|B_R(x_0)|e^{\delta(At_0)^\beta}
+C|B_R(x_0)|e^{-\delta(At_0)^\beta}\\
&\quad\le C|B_R(x_0)|.
\end{aligned}
\]
Here
\[
\delta(At_0)^\beta=\frac c2 K^\beta,
\]
so the last constant is independent of $\nu$.
This proves \eqref{eq:local-exponential}, with constants uniform
over all measures satisfying the stated maximal potential bound.
\end{proof}
\subsection{Wolff potential estimates for exponential reactions}
We next establish the Wolff potential estimates needed for the iteration later.
\begin{proposition}\label{prop:subabsorb}
 Under the assumptions of Theorem~\ref{thm:subcritical}, there are positive
 $b$ and $\varepsilon_*$ such that if \eqref{eq:sub-small} holds and
 $\omega=\mu+b\chi_{B_{2R}(x_0)}dx$, then
 \begin{equation}
 \W_{1,p}^\infty\!\left[
  H_l\!\left(\alpha(2c_pC_*\W_{1,p}^\infty[\omega])^\beta\right)
 \right]
 \le\W_{1,p}^\infty[\omega]
 \quad\text{in }\G,
 \label{eq:subabsorb}
 \end{equation}
 and the density inside the brackets belongs to $L^1(\G)$.
\end{proposition}
\begin{proof}
We shall choose $b=\varepsilon_*=\tau>0$, where $\tau$ will be
fixed at the end of the proof. Let
$\eta=(p-1)(\beta-1)/\beta$. For a measure satisfying
$\|\Mcal_p^\eta[\mu]\|_{L^\infty(\G)}\le\tau$, write
\[
\sigma=\tau^{-1}\mu+\chi_{B_{2R}(x_0)}dx,
\quad \omega=\tau\sigma.
\]
We first obtain bounds for $\sigma$ that are independent of
$\tau$ and $\mu$. For the background measure, the ball volume formula gives
\[
\frac{|B_t(x)\cap B_{2R}(x_0)|}{t^{Q-p}h_\eta(t)}
\le
\begin{cases}
|B_1(e)|t^p/h_\eta(t),&0<t\le2R,\\
|B_{2R}(x_0)|/(t^{Q-p}h_\eta(t)),&t>2R.
\end{cases}
\]
Both expressions are bounded on their respective ranges. Hence
\[
\|\Mcal_p^\eta[\chi_{B_{2R}(x_0)}dx]\|_{L^\infty(\G)}
\le C_2(\G,p,\eta,R),
\]
and therefore
\[
\|\Mcal_p^\eta[\sigma]\|_{L^\infty(\G)}
\le \tau^{-1}\|\Mcal_p^\eta[\mu]\|_{L^\infty(\G)}+C_2
\le1+C_2.
\]
Since
$\mu(\G)=\mu(B_r(x_0))$ for every $r>R$,
the definition of the maximal potential gives
\[
\frac{\mu(\G)}{r^{Q-p}h_\eta(r)}
=\frac{\mu(B_r(x_0))}{r^{Q-p}h_\eta(r)}
\le \Mcal_p^\eta[\mu](x_0)
\le\tau.
\]
Letting $r\downarrow R$, we obtain
\[
\tau^{-1}\mu(\G)\le R^{Q-p}h_\eta(R).
\]
Since $\sigma=\tau^{-1}\mu+\chi_{B_{2R}(x_0)}dx$ and $\mu\ge0$,
\[
\begin{split}
|B_{2R}(x_0)|
\le \sigma(\G)
=\tau^{-1}\mu(\G)+|B_{2R}(x_0)|
\le R^{Q-p}h_\eta(R)+|B_{2R}(x_0)|.
\end{split}
\]
Thus $0<m_0\le\sigma(\G)\le M_0$, where
\[
m_0=|B_{2R}(x_0)|,
\quad
M_0=R^{Q-p}h_\eta(R)+|B_{2R}(x_0)|.
\]
Moreover,
$\supp\sigma\subset\overline B_{2R}(x_0)$.
All constants below depend only on
$\G,p,l,\alpha,\beta,R,C_*$ and are independent of $\tau$ and $\mu$.

We next use these bounds to obtain an integrable reaction density.
Fix $T>\max\{4R,1\}$ depending only on $R$.
Applying Corollary~\ref{cor:localexp} at this fixed scale, we obtain
constants $\delta_0>0$ and $C_0<\infty$ such that
\[
\int_{B_T(x_0)}
\exp\!\left(\delta_0[\W_{1,p}^{\infty}[\sigma]]^\beta\right)dx
\le C_0.
\]
Choose
\[
q=\frac{2Q}{p}>\frac Qp,\quad
\delta=\frac{\delta_0}{q},
\]
and set
\[
g=H_l\!\left(\delta[\W_{1,p}^{\infty}[\sigma]]^\beta\right).
\]
Since $H_l(t)\le e^t$ for $t\ge0$ and $q\delta=\delta_0$, we have
\[
\int_{B_T(x_0)}g^q\,dx
\le
\int_{B_T(x_0)}
\exp\!\left(\delta_0[\W_{1,p}^{\infty}[\sigma]]^\beta\right)dx
\le C_0.
\]
H\"older's inequality then gives
\[
\int_{B_T(x_0)}g\,dx
\le |B_T(x_0)|^{1-1/q}
\left(\int_{B_T(x_0)}g^q\,dx\right)^{1/q}
\le C.
\]
To estimate $g$ outside $B_T(x_0)$, put
\[
s=\frac1{p-1},\quad
a_0=(Q-p)s,\quad
k_0=l\beta a_0.
\]
The assumption $l\beta>\frac{Q(p-1)}{Q-p}$ gives $k_0>Q$.
For $\rho=d(y,x_0)\ge T$, the support of $\sigma$ is disjoint
from $B_t(y)$ whenever $t<\rho-2R$. Consequently,
\begin{equation}
\begin{aligned}
\W_{1,p}^{\infty}[\sigma](y)
&\le M_0^s\int_{\rho-2R}^{\infty}
t^{-a_0}\frac{\dd t}{t}
=\frac{M_0^s}{a_0}(\rho-2R)^{-a_0}
\le C\rho^{-a_0},\\
\W_{1,p}^{\infty}[\omega](y)&=\tau^s\W_{1,p}^{\infty}[\sigma](y)\le\frac{\tau^sM_0^s}{a_0}(\rho-2R)^{-a_0}\le C\tau^s\rho^{-a_0}.
\end{aligned}
\label{eq:sub-tail}
\end{equation}
In particular,
\[
\delta[\W_{1,p}^{\infty}[\sigma](y)]^\beta
\le CT^{-\beta a_0}.
\]
Using $H_l(t)\le t^le^t/l!$, we obtain
\[
\begin{aligned}
g(y)
\le \frac{\delta^l}{l!}
[\W_{1,p}^{\infty}[\sigma](y)]^{l\beta}
\exp\!\left(\delta[\W_{1,p}^{\infty}[\sigma](y)]^\beta\right)
\le C\rho^{-k_0},
\quad \rho\ge T.
\end{aligned}
\]
Since $k_0>Q$ and $q>1$, the ball volume formula gives,
for $m=1,q$,
\[
\int_{\G\setminus B_T(x_0)}g^m\,dx
\le C\int_T^\infty r^{Q-1-mk_0}\,\dd r
\le C.
\]
Combining this with the estimates on $B_T(x_0)$, we have
\[
\|g\|_{L^1(\G)}+\|g\|_{L^q(\G)}\le C.
\]
We now use these two bounds to estimate the Wolff potential of $g$.
H\"older's inequality gives
\[
\int_{B_t(x)}g\,dx
\le |B_t(x)|^{1-1/q}\|g\|_{L^q(\G)}
\le Ct^{Q(1-1/q)}.
\]
For large $t$, we also have
\[
\int_{B_t(x)}g\,dx\le\|g\|_{L^1(\G)}\le C.
\]
Using the first estimate for $0<t<1$ and the second for $t\ge1$,
we obtain
\[
\begin{aligned}
\W_{1,p}^{\infty}[g](x)
&=\int_0^\infty
\left(\frac{\int_{B_t(x)}g\,dx}{t^{Q-p}}\right)^s
\frac{\dd t}{t}\\
&\le C\int_0^1t^{(p-Q/q)s}\frac{\dd t}{t}
+C\int_1^\infty t^{-a_0}\frac{\dd t}{t}
\le C.
\end{aligned}
\]
Both integrals converge because $q>Q/p$ and $p<Q$. 

We also need the decay of this potential at infinity.
Suppose $\rho=d(x,x_0)\ge2T$.
If $t<\rho/2$ and $y\in B_t(x)$, then
\[
d(y,x_0)\ge d(x,x_0)-d(x,y)>\rho/2\ge T.
\]
The pointwise estimate for $g$ therefore gives
\[
\int_{B_t(x)}g\,dx\le C\rho^{-k_0}t^Q.
\]
For $t\ge\rho/2$, the nonnegativity of $g$ and the preceding
$L^1(\G)$ bound give
\[
\int_{B_t(x)}g(y)\,dy
\le\int_\G g(y)\,dy
=\|g\|_{L^1(\G)}
\le C.
\]
It follows that
\[
\begin{aligned}
\W_{1,p}^{\infty}[g](x)
&\le C\rho^{-k_0s}
\int_0^{\rho/2}t^{ps}\frac{\dd t}{t}
+C\int_{\rho/2}^\infty t^{-a_0}\frac{\dd t}{t}\\
&\le C\rho^{(p-k_0)s}+C\rho^{-a_0}
\le C\rho^{-a_0}.
\end{aligned}
\]
The last inequality follows from $k_0>Q$ and $\rho>1$.

To compare these estimates with $\W_{1,p}^{\infty}[\sigma]$,
we use the fixed background measure.
For any $x\in\G$, write $\rho=d(x,x_0)$.
If $t>\rho+2R$, then
$B_{2R}(x_0)\subset B_t(x)$, and hence
$\sigma(B_t(x))\ge m_0$. Thus
\[
\W_{1,p}^{\infty}[\sigma](x)
\ge m_0^s\int_{\rho+2R}^{\infty}t^{-a_0}\frac{\dd t}{t}
=\frac{m_0^s}{a_0}(\rho+2R)^{-a_0}.
\]
On $B_{2T}(x_0)$, this gives a positive uniform lower bound,
while $\W_{1,p}^{\infty}[g]$ is uniformly bounded above.
For $\rho\ge2T$, it gives
$\W_{1,p}^{\infty}[\sigma](x)\ge c\rho^{-a_0}$,
which has the same decay as the upper bound for
$\W_{1,p}^{\infty}[g]$. Therefore
\[
\W_{1,p}^{\infty}[g](x)
\le C\W_{1,p}^{\infty}[\sigma](x),
\quad x\in\G.
\]
It remains to choose $\tau$ so that the required reaction
satisfies the absorption estimate.
Put $A=2c_pC_*$ and
\[
f=H_l\!\left(\alpha[A\W_{1,p}^{\infty}[\omega]]^\beta\right).
\]
Since $\omega=\tau\sigma$, homogeneity gives
\[
\alpha[A\W_{1,p}^{\infty}[\omega]]^\beta
=\alpha A^\beta\tau^{\beta s}
[\W_{1,p}^{\infty}[\sigma]]^\beta.
\]
Set
\[
\theta=\frac{\alpha A^\beta}{\delta}\tau^{\beta s},
\]
and take $\tau$ small enough that $\theta\le1$.
For $t\ge0$ and $0\le\theta\le1$,
\[
H_l(\theta t)
=\sum_{j=l}^\infty\frac{\theta^jt^j}{j!}
\le\theta^l\sum_{j=l}^\infty\frac{t^j}{j!}
=\theta^lH_l(t).
\]
Applying this with
$t=\delta[\W_{1,p}^{\infty}[\sigma]]^\beta$, we obtain
$f\le\theta^lg$. In particular, $f\in L^1(\G)$.
Moreover,
\[
\begin{aligned}
\W_{1,p}^{\infty}[f]
&\le\theta^{ls}\W_{1,p}^{\infty}[g]
\le C\theta^{ls}\W_{1,p}^{\infty}[\sigma]\\
&=C\theta^{ls}\tau^{-s}\W_{1,p}^{\infty}[\omega]
=C\left(\frac{\alpha A^\beta}{\delta}\right)^{ls}
\tau^{\frac{l\beta-(p-1)}{(p-1)^2}}
\W_{1,p}^{\infty}[\omega].
\end{aligned}
\]
Since
\[
l\beta>\frac{Q(p-1)}{Q-p}>p-1,
\]
the exponent of $\tau$ is positive.
We may therefore decrease $\tau$ so that the last coefficient
is at most one. With $b=\varepsilon_*=\tau$, this proves
\eqref{eq:subabsorb} and the asserted $L^1(\G)$ property.
\end{proof}
Throughout the critical case, we assume that $Q>1$.
\begin{proposition}\label{prop:criticalBM}
 Let $\nu$ be a nonnegative finite Radon measure supported in a ball
 $B_R(x_0)$.  For every $0<\delta<1$ there is $C_\delta=C_\delta(\G,Q)$
 such that, for every $z\in\G$ and $0<r\le R$,
 \begin{equation}
 \int_{B_r(z)}
 \exp\!\left(
 \frac{Q(1-\delta)\W_{1,Q}^{2r}[\nu\restriction B_{2r}(z)]}
 {[\nu(B_{2r}(z))]^{1/(Q-1)}}
 \right)dx\le C_\delta r^Q.
 \label{eq:critical-BM}
 \end{equation}
 The quotient is understood as zero if $\nu(B_{2r}(z))=0$.
\end{proposition}
\begin{proof}
We first prove the estimate on $B_1(e)$ for a probability measure
$\sigma$ on $\G$. Put $s=1/(Q-1)$.
To estimate its Wolff potential, define
\[
\mathcal A_\sigma(x)
=\sup_{t>0}\frac{\sigma(B_t(x))}{t^Q}.
\]
The $5r$ covering lemma, the ball volume formula and
$\sigma(\G)=1$ give
\[
|\{x\in\G:\mathcal A_\sigma(x)>\lambda\}|
\le\frac{C_\G}{\lambda},
\quad \lambda>0.
\]
In particular, $\mathcal A_\sigma(x)<\infty$ almost everywhere.
At each such point, we have
\[
\sigma(B_t(x))
\le1,
\quad
\sigma(B_t(x))
\le\mathcal A_\sigma(x)t^Q.
\]
If $\mathcal A_\sigma(x)\le2^{-Q}$, then
\[
\begin{aligned}
\W_{1,Q}^{2}[\sigma](x)
&=\int_0^2\sigma(B_t(x))^s\frac{\dd t}{t}
\le\mathcal A_\sigma(x)^s
\int_0^2t^{Qs}\frac{\dd t}{t}\\
&=\frac{(2^Q\mathcal A_\sigma(x))^s}{Qs}
\le\frac1{Qs}.
\end{aligned}
\]
If $\mathcal A_\sigma(x)>2^{-Q}$, we split the integral at
$t=\mathcal A_\sigma(x)^{-1/Q}<2$ and obtain
\[
\begin{aligned}
\W_{1,Q}^{2}[\sigma](x)
&\le\mathcal A_\sigma(x)^s
\int_0^{\mathcal A_\sigma(x)^{-1/Q}}
t^{Qs}\frac{\dd t}{t}
+\int_{\mathcal A_\sigma(x)^{-1/Q}}^2\frac{\dd t}{t}\\
&=\frac1{Qs}
+\frac1Q\log\!\left(2^Q\mathcal A_\sigma(x)\right).
\end{aligned}
\]
Thus, in both cases,
\[
\exp\!\left(Q(1-\delta)\W_{1,Q}^{2}[\sigma](x)\right)
\le e^{(1-\delta)/s}
\max\{1,2^Q\mathcal A_\sigma(x)\}^{1-\delta}.
\]
We next integrate this bound over $B_1(e)$.
The weak estimate for $\mathcal A_\sigma$ gives
\[
|\{x\in B_1(e):2^Q\mathcal A_\sigma(x)>\lambda\}|
\le\frac{C_\G}{\lambda},
\quad \lambda\ge1.
\]
Applying the layer cake formula and splitting the integral at $\lambda=1$, we obtain
\[
\begin{aligned}
&\int_{B_1(e)}
\max\{1,2^Q\mathcal A_\sigma(x)\}^{1-\delta}\,dx\\
&\quad=(1-\delta)\int_0^\infty
\lambda^{-\delta}
|\{x\in B_1(e):
\max\{1,2^Q\mathcal A_\sigma(x)\}>\lambda\}|
\,\dd\lambda\\
&\quad=(1-\delta)|B_1(e)|
\int_0^1\lambda^{-\delta}\,\dd\lambda\\
&\quad\quad
+(1-\delta)\int_1^\infty
\lambda^{-\delta}
|\{x\in B_1(e):2^Q\mathcal A_\sigma(x)>\lambda\}|
\,\dd\lambda\\
&\quad=|B_1(e)|
+(1-\delta)\int_1^\infty
\lambda^{-\delta}
|\{x\in B_1(e):2^Q\mathcal A_\sigma(x)>\lambda\}|
\,\dd\lambda\\
&\quad\le |B_1(e)|
+C_\G(1-\delta)\int_1^\infty
\lambda^{-1-\delta}\,\dd\lambda\\
&\quad=|B_1(e)|+\frac{C_\G(1-\delta)}{\delta}.
\end{aligned}
\]
Combining the two estimates, we obtain
\begin{equation}
\int_{B_1(e)}
\exp\!\left(Q(1-\delta)\W_{1,Q}^{2}[\sigma](\xi)\right)d\xi
\le C_\delta.
\label{eq:normalized-critical-BM}
\end{equation}
We now obtain the estimate on $B_r(z)$ by normalizing the mass
and rescaling the ball. Set
\[
m=\nu(B_{2r}(z)),
\quad
\nu_{z,r}=\nu\restriction B_{2r}(z).
\]
If $m=0$, then $\nu_{z,r}=0$. In this case, the exponential term
in \eqref{eq:critical-BM} is understood to be $1$.
The integral then equals $|B_r(z)|=|B_1(e)|r^Q$
and the conclusion follows from \eqref{eq:ball-volume}.
Then we can assume that $m>0$. Define
\[
T_{z,r}(y)=\delta_{1/r}(z^{-1}\circ y),
\quad
\sigma=m^{-1}(T_{z,r})_\#\nu_{z,r}.
\]
Then $\sigma$ is a probability measure concentrated in $B_2(e)$.
Write
\[
\xi=T_{z,r}(x),
\quad x=z\circ\delta_r(\xi).
\]
By the left invariance and homogeneity of the distance,
\[
T_{z,r}(B_t(x))=B_{t/r}(\xi).
\]
The definition of the pushforward measure therefore gives
\[
\nu_{z,r}(B_t(x))
=m\,\sigma(B_{t/r}(\xi)).
\]
Substituting this identity into the Wolff potential and then
setting $t=r\tau$, we obtain
\begin{equation}
\begin{aligned}
\W_{1,Q}^{2r}[\nu_{z,r}](z\circ\delta_r(\xi))
&=m^s\int_0^{2r}
\sigma(B_{t/r}(\xi))^s\frac{\dd t}{t}\\
&=m^s\int_0^2
\sigma(B_\tau(\xi))^s\frac{\dd\tau}{\tau}\\
&=m^{1/(Q-1)}\W_{1,Q}^{2}[\sigma](\xi).
\end{aligned}
\label{eq:critical-scaling-potential}
\end{equation}
Under the change of variables $x=z\circ\delta_r(\xi)$,
the ball $B_r(z)$ becomes $B_1(e)$ and $dx=r^Qd\xi$.
Applying \eqref{eq:normalized-critical-BM} and
\eqref{eq:critical-scaling-potential}, we obtain
\[
\begin{aligned}
&\int_{B_r(z)}
\exp\!\left(
\frac{Q(1-\delta)\W_{1,Q}^{2r}[\nu_{z,r}](x)}
{m^{1/(Q-1)}}
\right)dx\\
&\quad=r^Q\int_{B_1(e)}
\exp\!\left(Q(1-\delta)\W_{1,Q}^{2}[\sigma](\xi)\right)d\xi\\
&\quad\le C_\delta r^Q.
\end{aligned}
\]
This proves \eqref{eq:critical-BM}.
\end{proof}
\begin{proposition}\label{prop:criticalabsorb}
  Let $\Omega\Subset\G$, $D=\diam\Omega$, and
  $B_{10D}(x_0)\supset\Omega$.  If $l>Q-1$, there are
  $
   b=b(\G,Q,l,\alpha,D)>0,
   \varepsilon_0=\varepsilon_0(\G,Q,l,\alpha,D)>0,
  $
  such that if $\mu$ is a nonnegative finite Radon measure
in $\Omega$ satisfying
\[
\mu(\Omega)\le\varepsilon_0,
\]
then
 \begin{equation}
  \W_{1,Q}^{2D}\!\left[
   H_l\!\left(2\alpha C_*\W_{1,Q}^{2D}[\omega]\right)\chi_\Omega dx
  \right]
 \le\W_{1,Q}^{2D}[\omega]
 \quad\text{in }\Omega,
 \label{eq:criticalabsorb}
 \end{equation}
  where $\omega=\mu+b\chi_{B_{10D}(x_0)}dx$ and the density is extended by
  zero outside $\Omega$.  It belongs to $L^1(\Omega)$.
\end{proposition}

\begin{proof}
 Put
 \[
 s=\frac1{Q-1},\quad a=2\alpha C_*,\quad
 M=\omega(\G),\quad W=\W_{1,Q}^{2D}[\omega],\quad
 f=H_l(aW)\chi_\Omega.
 \]
 The case $M=0$ is immediate.  Suppose $M>0$ and set $V=W/M^s$.
 Apply Proposition~\ref{prop:criticalBM} with $\delta=1/2$ to the fixed
 integration ball $B_{10D}(x_0)$, taking $B_{11D}(x_0)$ as a support ball.
 Since $\omega\restriction B_{20D}(x_0)=\omega$ and
 $\W_{1,Q}^{2D}[\omega]\le\W_{1,Q}^{20D}[\omega]$, this gives
 \begin{equation}
 \int_\Omega e^{a_1V}\,dx\le C_\G D^Q,
 \quad a_1=\frac Q2.
 \label{eq:critical-fixed-exponential}
 \end{equation}
 Here and below, $C$ depends only on $\G,Q,l,\alpha$
and is independent of $D$. We first prove an $L^2$ bound for $f$. Since
$H_l(t)\le t^l e^t/l!$ for $t\ge0$, we have
\[
\int_\Omega f^2\,dx
\le \frac{a^{2l}M^{2ls}}{(l!)^2}
\int_\Omega V^{2l}e^{2aM^sV}\,dx.
\]
To apply the preceding exponential estimate, choose $M$ small
enough that
\begin{equation}
2aM^s\le\frac{a_1}{2}.
\label{eq:critical-mass-choice}
\end{equation}
Then
\[
V^{2l}e^{2aM^sV}
\le V^{2l}e^{a_1V/2}
=\bigl(V^{2l}e^{-a_1V/2}\bigr)e^{a_1V}.
\]
Since
\[
\sup_{v\ge0}v^{2l}e^{-a_1v/2}
=\left(\frac{4l}{e a_1}\right)^{2l}<\infty,
\]
it follows that
\[
V^{2l}e^{2aM^sV}\le C e^{a_1V}.
\]
Combining these estimates with the preceding bound for
$\int_\Omega e^{a_1V}\,dx$, we obtain
\begin{equation}
\|f\|_{L^2(\Omega)}^2
\le C M^{2ls}\int_\Omega e^{a_1V}\,dx
\le C D^Q M^{2ls}.
\label{eq:critical-density-L2}
\end{equation}
 We next estimate the Wolff potential of $f$. Extend $f$ by zero
outside $\Omega$. For any $x\in\G$ and $t>0$, H\"older's inequality gives
\[
\int_{B_t(x)}f\,dx
\le |B_t(x)|^{1/2}
\left(\int_{B_t(x)}f^2\,dx\right)^{1/2}
\le |B_1(e)|^{1/2}t^{Q/2}\|f\|_{L^2(\Omega)}.
\]
It follows that
\begin{equation}
\begin{split}
\W_{1,Q}^{2D}[f](x)
&=\int_0^{2D}
\left(\int_{B_t(x)}f\,dx\right)^s\frac{\dd t}{t}
\le C\|f\|_{L^2(\Omega)}^s
\int_0^{2D}t^{Qs/2}\frac{\dd t}{t}\\
&\le C D^{Qs/2}\|f\|_{L^2(\Omega)}^s
\le C D^{Qs}M^{ls^2},
\end{split}
\label{eq:critical-reaction-potential}
\end{equation}
where the last inequality follows from
\eqref{eq:critical-density-L2}. This estimate holds for every $x\in\G$. To compare this with $W$, write $m=\mu(\Omega)$ and
 $B=b|B_{10D}(x_0)|$, so that $M=m+B$.  If $x\in\Omega$ and
 $D<t<2D$, then $\Omega\subset B_t(x)$, then
 \[
 W(x)\ge\int_D^{2D}m^s\frac{\dd t}{t}=(\log2)m^s.
 \]
 Also, for $x\in B_{10D}(x_0)$ and $0<t\le2D$, the geodesic property
 implies
 \[
 |B_t(x)\cap B_{10D}(x_0)|\ge |B_{t/2}(e)|.
 \]
 Indeed, if $d(x,x_0)\ge t/2$, choose $y$ at distance $t/2$ from $x$
 along a minimizing geodesic to $x_0$; otherwise take $y=x_0$.
 In either case, $B_{t/2}(y)\subset B_t(x)\cap B_{10D}(x_0)$.
 Since
\[
\omega(B_t(x))=\mu(B_t(x))
+b|B_t(x)\cap B_{10D}(x_0)|\ge b|B_t(x)\cap B_{10D}(x_0)|,
\]
the preceding volume estimate gives, for $x\in\Omega$,
\[
\begin{split}
W(x)
&=\int_0^{2D}\omega(B_t(x))^s\frac{\dd t}{t}
\ge b^s\int_0^{2D}
|B_t(x)\cap B_{10D}(x_0)|^s\frac{\dd t}{t}\\
&\ge b^s|B_1(e)|^s2^{-Qs}
\int_0^{2D}t^{Qs}\frac{\dd t}{t}
=\frac{10^{-Qs}}{Qs}\,B^s.
\end{split}
\]
 Combining these two lower bounds and using
 $\max\{m,B\}\ge M/2$ yields $W(x)\ge cM^s$ on $\Omega$, with
 $c>0$ depending only on $Q$.  It follows from
 \eqref{eq:critical-reaction-potential} that
 \begin{equation}
 \W_{1,Q}^{2D}[f](x)
 \le C D^{Q/(Q-1)}
 M^{\frac{l-Q+1}{(Q-1)^2}}W(x),\quad x\in\Omega.
 \label{eq:critical-absorption-factor}
 \end{equation}
 The exponent of $M$ is positive because $l>Q-1$.
 Fix a constant $M_*>0$, depending only on $\G,Q,l,\alpha,D$, so that
 \eqref{eq:critical-mass-choice} holds and the coefficient in
 \eqref{eq:critical-absorption-factor} is at most $1$ whenever
 $M\le M_*$.  Choose $b>0$ first so that
 $b|B_{10D}(x_0)|\le M_*/2$, and then choose
 $0<\varepsilon_0\le M_*/2$.  Thus
 \[
 M\le\varepsilon_0+b|B_{10D}(x_0)|\le M_*,
 \]
 proving \eqref{eq:criticalabsorb} and the required integrability.
\end{proof}

\section{Iteration on bounded open sets}

We now prove the iteration lemma.

\begin{lemma}\label{lem:iteration}
 Let $D\Subset\G$ be a bounded open set, let $\nu$ be a nonnegative finite measure in
 $D$, and suppose that a nonnegative measure $\omega$ satisfies
 $\nu\le\omega$ and
 \begin{equation}
 \W_{1,p}^{2\diam D}\!\left[
  P_{l,\alpha,\beta}\bigl(2c_pC_*\W_{1,p}^\infty[\omega]\chi_D\bigr)
 \right]
 \le\W_{1,p}^\infty[\omega]\quad\text{in }D.
 \label{eq:iteration-hyp}
 \end{equation}
 Assume also that
\begin{equation}
P_{l,\alpha,\beta}
(2c_pC_*\W_{1,p}^{\infty}[\omega]\chi_D)
\in L^1(D).
\label{integrability condition}
\end{equation} 
Then the Dirichlet problem
 \begin{equation}
 \begin{cases}
  -\Delta_{\G,p}v=P_{l,\alpha,\beta}(v)+\nu&\text{in }D,\\
 v=0&\text{on }\partial D
 \end{cases}
 \label{eq:boundediteration}
 \end{equation}
 has a nonnegative solution satisfying
 \begin{equation}
  v\le2c_pC_*\W_{1,p}^\infty[\omega]\quad\text{in }D.
 \label{eq:iterationbound}
 \end{equation}
 For $p=Q$ and $\beta=1$, the lemma remains true if $\W_{1,p}^\infty[\omega]$ in \eqref{eq:iteration-hyp}--\eqref{eq:iterationbound}
 is replaced by $\W_{1,Q}^{2\diam D}[\omega]$ and $c_p$ is replaced by $1$.
\end{lemma}

\begin{proof} 
By \cite[Theorem~4.1]{PVTAMS}, we can choose a nonnegative measure data
solution $v_0$ of
\[
-\Delta_{\G,p}v_0=\nu\quad\text{in }D,
\quad v_0=0\quad\text{on }\partial D,
\]
which satisfies
\[
0\le v_0
\le C_*\W_{1,p}^{2\diam D}[\nu]
\le 2c_pC_*\W_{1,p}^{\infty}[\omega],
\]
and there exists a sequence of
nonnegative functions $v_{0,n}\in S^{1,p}_0(D)$ satisfying
\[
-\Delta_{\G,p}v_{0,n}
=\rho_n*\nu
\quad\text{in }D,
\]
such that, after passing to a subsequence, $v_{0,n}\to v_0$
almost everywhere in $D$. Here $\rho_n(x)=n^Q\rho(\delta_nx)$, where
$\delta_n$ is given by \eqref{eq:dilation} with $r=n$, and $\rho\in C_c^\infty(B_1(e))$ is nonnegative and satisfies
$\int_{\G}\rho\,dx=1$.
Here the data are extended by zero outside $D$. 
By the integrability condition in \eqref{integrability condition},
\[
0\le P_{l,\alpha,\beta}(v_0)
\le P_{l,\alpha,\beta}
   \bigl(2c_pC_*\W_{1,p}^{\infty}[\omega]\bigr)
\in L^1(D).
\]
Thus $P_{l,\alpha,\beta}(v_0)\chi_D\,dx+\nu$ is a finite
nonnegative measure. Applying \cite[Theorem~4.1]{PVTAMS} to
this measure, we obtain a nonnegative solution
$v_1$ of
\[
-\Delta_{\G,p}v_1=P_{l,\alpha,\beta}(v_0)+\nu
\quad\text{in }D,
\quad v_1=0\quad\text{on }\partial D.
\]
Here the data are extended by zero outside $D$.
Moreover, \cite[Theorem~4.1]{PVTAMS} shows that we can choose $v_1$ together with a sequence of
nonnegative functions $v_{1,n}\in S^{1,p}_0(D)$ satisfying
\[
-\Delta_{\G,p}v_{1,n}
=\rho_n*\bigl(P_{l,\alpha,\beta}(v_0)\chi_D\,dx+\nu\bigr)
\quad\text{in }D,
\]
such that, after passing to a subsequence, $v_{1,n}\to v_1$ almost everywhere in $D$. Define
\[
\varphi_n=\min\{v_{1,n}-v_{0,n},0\}\in S^{1,p}_0(D).
\]
Testing the two equations with $\varphi_n$ and subtracting,
we obtain
\[
\begin{aligned}
0
&\le \int_{\{v_{1,n}<v_{0,n}\}}
\bigl(|Xv_{1,n}|^{p-2}Xv_{1,n}
      -|Xv_{0,n}|^{p-2}Xv_{0,n}\bigr)
\cdot(Xv_{1,n}-Xv_{0,n})\,dx\\
&=\int_D
\bigl(|Xv_{1,n}|^{p-2}Xv_{1,n}
      -|Xv_{0,n}|^{p-2}Xv_{0,n}\bigr)
\cdot X\varphi_n\,dx\\
&=\int_D
\bigl[\rho_n*(P_{l,\alpha,\beta}(v_0)\chi_D)\bigr]
\varphi_n\,dx
\le0.
\end{aligned}
\]
Thus $\varphi_n=0 \,\operatorname{a.e.}$, which means that $v_{0,n}\le v_{1,n}$ almost everywhere in $D$. Passing to the limit along the common subsequence, we obtain $v_0\le v_1$ almost everywhere. Taking their $p$-superharmonic representatives, we obtain $v_0\le v_1$ everywhere in $D$.
Moreover, the Wolff potential estimate and
\eqref{eq:iteration-hyp} give
\begin{equation}
\begin{aligned}
v_1
&\le C_*\W_{1,p}^{2\diam D}
   [P_{l,\alpha,\beta}(v_0)\chi_D+\nu]\\
&\le c_pC_*\Bigl(
   \W_{1,p}^{2\diam D}
   [P_{l,\alpha,\beta}
      (2c_pC_*\W_{1,p}^{\infty}[\omega])\chi_D]
   +\W_{1,p}^{2\diam D}[\nu]\Bigr)\\
&\le 2c_pC_*\W_{1,p}^{\infty}[\omega].
\end{aligned}
\label{v1 estimate}
\end{equation}
Hence
\[
0\le v_0\le v_1
\le 2c_pC_*\W_{1,p}^{\infty}[\omega].
\]
Let $j\ge2$ and suppose that $v_0,\ldots,v_{j-1}$ have been
constructed with
\[
0\le v_0\le\cdots\le v_{j-1}
\le 2c_pC_*\W_{1,p}^{\infty}[\omega].
\]
The integrability assumption \eqref{integrability condition} gives
\[
0\le P_{l,\alpha,\beta}(v_{j-1})
\le P_{l,\alpha,\beta}
   \bigl(2c_pC_*\W_{1,p}^{\infty}[\omega]\bigr)
\in L^1(D).
\]
Thus $P_{l,\alpha,\beta}(v_{j-1})\chi_D\,dx+\nu$ is a finite
nonnegative measure. By \cite[Theorem~4.1]{PVTAMS}, the problem
\[
-\Delta_{\G,p}v_j=P_{l,\alpha,\beta}(v_{j-1})+\nu
\quad\text{in }D,
\quad v_j=0\quad\text{on }\partial D
\]
admits a nonnegative solution. Since $v_{j-2}\le v_{j-1}$,
\[
P_{l,\alpha,\beta}(v_{j-1})
-P_{l,\alpha,\beta}(v_{j-2})\ge0.
\]
For $j\ge2$, let $v_{j,n}\in S^{1,p}_0(D)$ be the solution of
\[
-\Delta_{\G,p}v_{j,n}
=\rho_n*\bigl(P_{l,\alpha,\beta}(v_{j-1})\chi_D\,dx+\nu\bigr)
\quad\text{in }D.
\] such that, after passing to a further subsequence,
$v_{j,n}\to v_j$ almost everywhere in $D$. The same calculation with
$\min\{v_{j,n}-v_{j-1,n},0\}$ gives
$v_{j-1,n}\le v_{j,n}\,\,\text{a.e. in }D$. Passing to the common subsequence yields
$
v_{j-1}\le v_j \,\,\text{a.e. in }D.
$
Since these functions are taken with their $p$-superharmonic representatives,
the inequality holds everywhere in $D$. Moreover, the same estimate as \eqref{v1 estimate} gives
\[
v_j
\le C_*\W_{1,p}^{2\diam D}
   [P_{l,\alpha,\beta}(v_{j-1})\chi_D+\nu]
\le 2c_pC_*\W_{1,p}^{\infty}[\omega].
\]
This completes the induction and yields
\[
0\le v_0\le v_1\le\cdots
\le 2c_pC_*\W_{1,p}^{\infty}[\omega].
\]
and
\begin{equation}
\label{P L1}
0\le P_{l,\alpha,\beta}(v_{j})
\le P_{l,\alpha,\beta}
   \bigl(2c_pC_*\W_{1,p}^{\infty}[\omega]\bigr)
\in L^1(D). \quad\forall j\geq1
\end{equation}
This proves the uniform bound. 

If $p=Q$ and $\beta=1$, similar as before we have 
  \[ 
  v_{j+1}\le C_*\W_{1,Q}^{2\diam D} 
  [H_l(\alpha v_j)\chi_D+\nu]. 
  \] 
  Since $c_Q=1$, the same calculation, with the assumption of the critical case of \eqref{eq:iteration-hyp} and \eqref{integrability condition} gives 
  \[ 
  v_{j+1}\le2C_*\W_{1,Q}^{2\diam D}[\omega]. 
  \] 
Moreover,
\[
0\le H_l(\alpha v_j)
\le H_l\bigl(2\alpha C_*
   \W_{1,Q}^{2\diam D}[\omega]\bigr)
\in L^1(D).
\]
The following argument applies to both cases.

Let $v=\lim_jv_j$.
The preceding bounds imply that $v$ is finite almost everywhere
in $D$. Since $\{v_j\}$ is an increasing sequence of nonnegative
$p$-superharmonic functions, $v$ is $p$-superharmonic.
Using \eqref{P L1} and the dominated convergence theorem we have
\[
P_{l,\alpha,\beta}(v_j)\to P_{l,\alpha,\beta}(v)
\quad\text{in }L^1(D).
\]
On the other hand, since $v_j$ is a nonnegative $p$-superharmonic function and $v$ is $p$-superharmonic, \cite[Theorem~3.3]{PVTAMS} gives
$-\Delta_{\G,p}v_j\rightharpoonup-\Delta_{\G,p}v$.
Passing to the limit in
$-\Delta_{\G,p}v_{j+1}=P_{l,\alpha,\beta}(v_j)\,dx+\nu$
gives the equation in \eqref{eq:boundediteration}. 

Fix $j\ge0$ and $k>0$. Testing the equation for $v_{j+1,n}$
with $v_{j+1,n}\wedge k\in S^{1,p}_0(D)$, we obtain
\[
\begin{aligned}
\int_D |X(v_{j+1,n}\wedge k)|^p\,dx
&=\int_D |Xv_{j+1,n}|^{p-2}Xv_{j+1,n}
   \cdot X(v_{j+1,n}\wedge k)\,dx\\
&=\int_D (v_{j+1,n}\wedge k)
   \bigl[\rho_n*
   (P_{l,\alpha,\beta}(v_j)\chi_D\,dx+\nu)\bigr]\,dx\\
&\le k\int_{\G}
   \bigl[\rho_n*
   (P_{l,\alpha,\beta}(v_j)\chi_D\,dx+\nu)\bigr]\,dx\\
&=k\left(
   \nu(D)+\int_D P_{l,\alpha,\beta}(v_j)\,dx
   \right),
\end{aligned}
\]
where we used the nonnegativity of the data and
$\int_{\G}\rho_n\,dx=1$.

Since $v_{j+1,n}\to v_{j+1}$ almost everywhere in $D$
and $0\le v_{j+1,n}\wedge k\le k$, dominated convergence gives
\[
v_{j+1,n}\wedge k\to v_{j+1}\wedge k
\quad\text{in }L^p(D).
\]
The preceding estimate and
$\|v_{j+1,n}\wedge k\|_{L^p(D)}^p\le k^p|D|$
give boundedness in $S^{1,p}_0(D)$. Hence, after passing
to a further subsequence,
\[
v_{j+1,n}\wedge k\rightharpoonup v_{j+1}\wedge k
\quad\text{in }S^{1,p}_0(D).
\]
By weak lower semicontinuity,
\[
\begin{aligned}
\int_D |X(v_{j+1}\wedge k)|^p\,dx
&\le\liminf_{n\to\infty}
   \int_D |X(v_{j+1,n}\wedge k)|^p\,dx\\
&\le k\left(
   \nu(D)+\int_D P_{l,\alpha,\beta}(v_j)\,dx
   \right)\\
&\le Ck,
\end{aligned}
\]
where
$C=\nu(D)+\int_D P_{l,\alpha,\beta}(v)\,dx<\infty$,
since $v_j\le v$ and $P_{l,\alpha,\beta}$ is nondecreasing.
In particular, $C$ is independent of $j$ and $k$. Together with $0\le v_{j+1}\wedge k\le k$, this shows that 
$\{v_{j+1}\wedge k\}_j$ is bounded in $S^{1,p}_0(D)$ 
for each fixed $k>0$. 
Since $S^{1,p}_0(D)$ is reflexive, after passing to a subsequence, there exists a function
$h\in S^{1,p}_0(D)$ such that
\[
v_{j+1}\wedge k\rightharpoonup h
\quad\text{in }S^{1,p}_0(D).
\]
On the other hand, dominated convergence gives
$v_{j+1}\wedge k\to v\wedge k$ in $L^p(D)$.
Thus $h=v\wedge k$, and hence
\[
v\wedge k\in S^{1,p}_0(D).
\]
Since $k>0$ was arbitrary, the zero boundary condition holds,
which completes the proof of \eqref{eq:boundediteration}.
\end{proof}

\section{Proof of the main theorems}
\begin{proof}[Proof of Theorem~\ref{thm:subcritical}]
 Choose a bounded open exhaustion $\Omega_1\Subset\Omega_2\Subset\cdots\Subset\Omega$
  with $K_\mu\subset\Omega_1$ and $\bigcup_j\Omega_j=\Omega$.  Proposition
 \ref{prop:subabsorb} gives \eqref{eq:subabsorb}; Lemma~\ref{lem:iteration}
 applied to $D=\Omega_j$ produces a solution $u_j$ of \eqref{eq:subproblem}
 in $\Omega_j$ with zero boundary values and with the uniform upper bound
 \begin{equation}
  u_j\le2c_pC_*\W_{1,p}^\infty[\omega].
 \label{eq:exhaustionbound}
 \end{equation}
 By an approximation and comparison argument analogous to those in
\cite{PhucVerbitsky2008,NGARMA2014,PVTAMS}, 
we may choose the solutions so that
$u_j\le u_{j+1}$ in $\Omega_j$. Hence $u=\lim_ju_j$ is finite almost everywhere by
  \eqref{eq:exhaustionbound}.The monotone limit $u$ is $p$-superharmonic.

  Let $\varphi\in C_c^\infty(\Omega)$.  For $j$ sufficiently large,
  $\supp\varphi\subset\Omega_j$.  The integrable majorant in
  \eqref{eq:subbound} and monotone convergence show that
 \[
  P_{l,\alpha,\beta}(u_j)\to P_{l,\alpha,\beta}(u)
 \quad\text{in }L^1(\supp\varphi).
 \]
  The weak continuity theorem of \cite[Theorem~3.1(ii)]{TW} gives weak convergence of the
  Riesz measures. Since
   $P_{l,\alpha,\beta}(u_j)dx+\mu\rightharpoonup
   P_{l,\alpha,\beta}(u)dx+\mu$, the limit satisfies
  \eqref{eq:subproblem}.
Set
\[
 U=2c_pC_*\W^\infty_{1,p}[\omega],
 \quad
 \mathcal M=\mu(\G)+
 \int_{\G}P_{l,\alpha,\beta}(U)\,dx<\infty.
\]
By the estimate in the proof of Lemma~\ref{lem:iteration},
we have
\[
\begin{aligned}
\int_{\Omega_j}|X(u_j\wedge k)|^p\,dx
&\le k\left(
\mu(\Omega_j)+
\int_{\Omega_j}P_{l,\alpha,\beta}(u_j)\,dx
\right)\\
&\le k\mathcal M,
\quad k>0.
\end{aligned}
\]

Fix $k>0$ and $\zeta\in C_c^\infty(\G)$.
Extend $u_j\wedge k$ by zero from $\Omega_j$ to $\Omega$,
and denote this extension by $w_j$.
Then $\zeta w_j\in S_0^{1,p}(\Omega)$, and
\[
 \|\zeta w_j\|_{L^p(\Omega)}^p
 \le k^p\|\zeta\|_{L^p(\G)}^p,
\]
\[
 \|X(\zeta w_j)\|_{L^p(\Omega)}^p
 \le 2^{p-1}\left(
 \|\zeta\|_\infty^p k\mathcal M
 +k^p\|X\zeta\|_{L^p(\G)}^p
 \right).
\]
Thus $\zeta w_j$ is bounded in $S_0^{1,p}(\Omega)$.
Since $w_j\to u\wedge k$ almost everywhere and
$|\zeta w_j|\le k|\zeta|$, dominated convergence gives
$\zeta w_j\to\zeta(u\wedge k)$ strongly in $L^p(\Omega)$. Since $S^{1,p}_0(\Omega)$ is reflexive, by reflexivity and the preceding boundedness, there exist a subsequence and a function $h\in S_0^{1,p}(\Omega)$ such that
\[
\zeta w_j\rightharpoonup h
\quad\text{in }S^{1,p}_0(\Omega).
\]
The strong $L^p(\Omega)$ convergence implies
$h=\zeta(u\wedge k)$. Hence
\[
\zeta(u\wedge k)\in S^{1,p}_0(\Omega).
\]
 The bound \eqref{eq:subbound} follows from \eqref{eq:exhaustionbound}.
  When $\Omega=\G$, the decay in \eqref{eq:sub-tail} implies that the right side
 of \eqref{eq:subbound} tends to zero along $d(x,x_0)\to\infty$; hence
 $\inf_\G u=0$.
\end{proof}
\begin{proof}[Proof of Theorem~\ref{thm:critical}]
Proposition~\ref{prop:criticalabsorb} gives
\eqref{eq:criticalabsorb} and
\[
H_l\!\left(2\alpha C_*
\W_{1,Q}^{2\diam\Omega}[\omega]\right)\in L^1(\Omega).
\]
Since $\mu\le\omega$, the critical version of
Lemma~\ref{lem:iteration}, applied on $\Omega$ with $\nu=\mu$,
gives a nonnegative solution $u$ of
\eqref{eq:criticalproblem} satisfying \eqref{eq:criticalbound}.
\end{proof}

\section{Further remarks}
\begin{remark}
 The condition $l\beta>Q(p-1)/(Q-p)$ in Theorem~\ref{thm:subcritical}
 makes $H_l(c[\W_{1,p}^\infty[\omega]]^\beta)$ integrable at infinity.
 It also implies $l\beta>p-1$, which makes the small parameter exponent
 in the absorption proof positive.  The critical theorem has no
 whole space tail; there $l>Q-1$ makes the exponent of $M$ in
 \eqref{eq:critical-absorption-factor} positive.
\end{remark}

\begin{remark}
An extension of Theorem~\ref{thm:critical} to $H_l(\alpha u^\beta)$ with
$\beta>1$ cannot retain a small total mass hypothesis alone while allowing
arbitrary measures. Indeed, already in the Euclidean critical case, an atom
of mass $m>0$ at $x_0$ forces the local lower bound
$u(x)\ge c m^{1/(Q-1)}\log(r_0/|x-x_0|)$ for a sufficiently small fixed
$r_0>0$, by the lower Wolff estimate in \cite{PVTAMS,TW}.
Consequently, $H_l(\alpha u^\beta)$ is not locally integrable near $x_0$.
Such an extension would require additional restrictions on the data.
\end{remark}

\section{Acknowledgement}
The first named author is supported by the Natural
Science Foundation of Tianjin, No. 22JCJQJC00130 and Fundamental Research Funds for the Central Universities. The second named author is supported by China Scholarship Council,No.202506200090.

\bibliographystyle{abbrvnat}
\bibliography{paper}

@book{AdamsHedberg1996,
  author = {Adams, D. R. and Hedberg, L. I.},
  title = {{Function Spaces and Potential Theory}},
  series = {Grundlehren der Mathematischen Wissenschaften},
  volume = {314},
  publisher = {Springer},
  address = {Berlin},
  year = {1996},
  doi = {10.1007/978-3-662-03282-4}
}

@article{BaloghManfrediTyson2003,
  author = {Balogh, Z. M. and Manfredi, J. J. and Tyson, J. T.},
  title = {Fundamental solution for the {$Q$}-Laplacian and sharp {Moser--Trudinger} inequality in {Carnot} groups},
  journal = {J. Funct. Anal.},
  volume = {204},
  number = {1},
  pages = {35--49},
  year = {2003},
  doi = {10.1016/S0022-1236(02)00169-6}
}

@article{BoccardoGallouet1989,
  author = {Boccardo, L. and Gallou{\"e}t, T.},
  title = {Nonlinear elliptic and parabolic equations involving measure data},
  journal = {J. Funct. Anal.},
  volume = {87},
  number = {1},
  pages = {149--169},
  year = {1989},
  doi = {10.1016/0022-1236(89)90005-0}
}

@article{BoccardoGallouetOrsina1996,
  author = {Boccardo, L. and Gallou{\"e}t, T. and Orsina, L.},
  title = {Existence and uniqueness of entropy solutions for nonlinear elliptic equations with measure data},
  journal = {Ann. Inst. H. Poincar{\'e} C Anal. Non Lin{\'e}aire},
  volume = {13},
  number = {5},
  pages = {539--551},
  year = {1996},
  doi = {10.1016/S0294-1449(16)30113-5}
}

@book{BLU,
  author = {Bonfiglioli, A. and Lanconelli, E. and Uguzzoni, F.},
  title = {{Stratified Lie Groups and Potential Theory for Their Sub-Laplacians}},
  series = {Springer Monographs in Mathematics},
  publisher = {Springer},
  address = {Berlin},
  year = {2007}
}

@article{Brezis-Merle,
  author = {Brezis, H. and Merle, F.},
  title = {Uniform estimates and blow-up behavior for solutions of {$-\Delta u=V(x)e^u$} in two dimensions},
  journal = {Comm. Partial Differential Equations},
  volume = {16},
  number = {8-9},
  pages = {1223--1253},
  year = {1991}
}

@article{Covering,
  author = {Conde-Alonso, J. M. and {Di Plinio}, F. and Parissis, I. and Vempati, M. N.},
  title = {A metric approach to sparse domination},
  journal = {Ann. Mat. Pura Appl. (4)},
  volume = {201},
  number = {4},
  pages = {1639--1675},
  year = {2022},
  url = {https://doi.org/10.1007/s10231-021-01174-7},
  doi = {10.1007/s10231-021-01174-7}
}

@article{Folland1975,
  author = {Folland, G. B.},
  title = {Subelliptic estimates and function spaces on nilpotent {Lie} groups},
  journal = {Ark. Mat.},
  volume = {13},
  pages = {161--207},
  year = {1975},
  doi = {10.1007/BF02386204}
}

@book{FollandStein1982,
  author = {Folland, G. B. and Stein, E. M.},
  title = {{Hardy Spaces on Homogeneous Groups}},
  series = {Mathematical Notes},
  volume = {28},
  publisher = {Princeton University Press},
  address = {Princeton, NJ},
  year = {1982}
}

@article{KTM-ASNSP,
  author = {Kilpel{\"a}inen, T. and Mal{\'y}, J.},
  title = {Degenerate elliptic equations with measure data and nonlinear potentials},
  journal = {Ann. Scuola Norm. Sup. Pisa Cl. Sci. (4)},
  volume = {19},
  number = {4},
  pages = {591--613},
  year = {1992},
  url = {http://www.numdam.org/item?id=ASNSP_1992_4_19_4_591_0}
}

@article{KM2,
  author = {Kilpel{\"a}inen, T. and Mal{\'y}, J.},
  title = {The {Wiener} test and potential estimates for quasilinear elliptic equations},
  journal = {Acta Math.},
  volume = {172},
  number = {1},
  pages = {137--161},
  year = {1994},
  url = {https://doi.org/10.1007/BF02392793},
  doi = {10.1007/BF02392793}
}

@article{KuusiMingione2014,
  author = {Kuusi, T. and Mingione, G.},
  title = {Guide to nonlinear potential estimates},
  journal = {Bull. Math. Sci.},
  volume = {4},
  number = {1},
  pages = {1--82},
  year = {2014},
  doi = {10.1007/s13373-013-0048-9}
}

@book{MarcusVeron2014,
  author = {Marcus, M. and V{\'e}ron, L.},
  title = {{Nonlinear Second Order Elliptic Equations Involving Measures}},
  series = {De Gruyter Series in Nonlinear Analysis and Applications},
  volume = {21},
  publisher = {De Gruyter},
  address = {Berlin},
  year = {2014},
  doi = {10.1515/9783110305319}
}

@article{NGARMA2014,
  author = {Nguyen, Q.-H. and V{\'e}ron, L.},
  title = {Quasilinear and {Hessian} type equations with exponential reaction and measure data},
  journal = {Arch. Ration. Mech. Anal.},
  volume = {214},
  number = {1},
  pages = {235--267},
  year = {2014},
  doi = {10.1007/s00205-014-0756-7}
}

@article{MaWang2026,
  author = {Ma, Shiguang and Wang, Zijian},
  title = {{$N$}-Laplacian and {$N/2$}-Hessian Type Equations with Exponential Reaction Terms and Measure Data},
  journal = {Potential Anal.},
  volume = {64},
  pages = {Article 27},
  year = {2026},
  doi = {10.1007/s11118-025-10266-5}
}

@article{PhucVerbitsky2008,
  author = {Phuc, N. C. and Verbitsky, I. E.},
  title = {Quasilinear and {Hessian} equations of {Lane--Emden} type},
  journal = {Ann. of Math. (2)},
  volume = {168},
  number = {3},
  pages = {859--914},
  year = {2008},
  doi = {10.4007/annals.2008.168.859}
}

@article{PVTAMS,
  author = {Phuc, N. C. and Verbitsky, I. E.},
  title = {Quasilinear equations with source terms on {Carnot} groups},
  journal = {Trans. Amer. Math. Soc.},
  volume = {365},
  number = {12},
  pages = {6569--6593},
  year = {2013},
  url = {https://doi.org/10.1090/S0002-9947-2013-05920-X},
  doi = {10.1090/S0002-9947-2013-05920-X}
}

@article{TW,
  author = {Trudinger, N. S. and Wang, X.-J.},
  title = {On the weak continuity of elliptic operators and applications to potential theory},
  journal = {Amer. J. Math.},
  volume = {124},
  number = {2},
  pages = {369--410},
  year = {2002},
  url = {http://muse.jhu.edu/journals/american_journal_of_mathematics/v124/124.2trudinger.pdf}
}

\end{document}